\documentclass[12pt,leqno,draft]{article}
\usepackage{amsfonts}
\usepackage{amsmath, amsthm, amsfonts, amssymb, color}
\usepackage{mathrsfs}
\usepackage{color}
\usepackage[notcite,notref]{showkeys}

\newtheorem{thm}{Theorem}[section]

\newtheorem{lem}[thm]{Lemma}
\newtheorem{prp}[thm]{Proposition}

\newtheorem{rem}[thm]{Remark}
\theoremstyle{definition}
\newtheorem{defn}{Definition}[section]
\newcommand{\scr}[1]{\mathscr #1}
\definecolor{wco}{rgb}{0.5,0.2,0.3}

\numberwithin{equation}{section} \theoremstyle{remark}

\newcommand{\ua}{\uparrow}

\title{{\bf
  Asymptotic Log-Harnack Inequality for Degenerate SPDEs  with Reflection}\footnote{
 Supported in
 part by  National Key R\&D Program of China (No. 2022YFA1006000, 2022YFA1006001) and NNSFC (12131019, 12371151, 12426655, 12531007, 12571158). } }
\author{
{\bf   Qi Li$^{(a)}$,  Feng-Yu Wang$^{(b)}$,  Tu-Sheng Zhang$^{(a)}$}\\
\footnotesize{$^{a)}$   School of Mathematical Sciences, University of Science and Technology of China, Hefei,  China.
}\\
\footnotesize{$^{b)}$  Center for Applied Mathematics and KL-AAGDM, Tianjin University, Tianjin 300072, China}\\
\footnotesize{vivien777@mail.ustc.edu.cn,  wangfy@tju.edu.cn, Tusheng.Zhang@manchester.ac.uk}}
\begin{document}
\allowdisplaybreaks
\def\R{\mathbb R}  \def\ff{\frac} \def\ss{\sqrt} \def\B{\mathbf
B} \def\W{\mathbb W}
\def\N{\mathbb N} \def\kk{\kappa} \def\m{{\bf m}}
\def\ee{\varepsilon}\def\ddd{D^*}
\def\dd{\delta} \def\DD{\Delta} \def\vv{\varepsilon} \def\rr{\rho}
\def\<{\langle} \def\>{\rangle} \def\GG{\Gamma} \def\gg{\gamma}
  \def\nn{\nabla} \def\pp{\partial} \def\E{\mathbb E}
\def\d{\text{\rm{d}}} \def\bb{\beta} \def\aa{\alpha} \def\D{\scr D}
  \def\si{\sigma} \def\ess{\text{\rm{ess}}}
\def\beg{\begin} \def\beq{\begin{equation}}  \def\F{\scr F}
\def\Ric{\text{\rm{Ric}}} \def\Hess{\text{\rm{Hess}}}
\def\e{\text{\rm{e}}} \def\ua{\underline a} \def\OO{\Omega}  \def\oo{\omega}
 \def\tt{\tilde} \def\Ric{\text{\rm{Ric}}}
\def\cut{\text{\rm{cut}}} \def\P{\mathbb P} \def\ifn{I_n(f^{\bigotimes n})}
\def\C{\scr C}      \def\aaa{\mathbf{r}}     \def\r{r}
\def\gap{\text{\rm{gap}}} \def\prr{\pi_{{\bf m},\varrho}}  \def\r{\mathbf r}
\def\Z{\mathbb Z} \def\vrr{\varrho} \def\ll{\lambda}
\def\L{\scr L}\def\Tt{\tt} \def\TT{\tt}\def\II{\mathbb I}
\def\i{{\rm in}}\def\Sect{{\rm Sect}}  \def\H{\mathbb H}
\def\M{\scr M}\def\Q{\mathbb Q} \def\texto{\text{o}} \def\LL{\Lambda}
\def\Rank{{\rm Rank}} \def\B{\scr B} \def\i{{\rm i}} \def\HR{\hat{\R}^d}
\def\to{\rightarrow}\def\l{\ell}\def\iint{\int}
\def\EE{\scr E}\def\Cut{{\rm Cut}}
\def\A{\scr A} \def\Lip{{\rm Lip}}
\def\BB{\scr B}\def\Ent{{\rm Ent}}\def\L{\scr L}
\def\R{\mathbb R}  \def\ff{\frac} \def\ss{\sqrt} \def\B{\mathbf
B}
\def\N{\mathbb N} \def\kk{\kappa} \def\m{{\bf m}}
\def\dd{\delta} \def\DD{\Delta} \def\vv{\varepsilon} \def\rr{\rho}
\def\<{\langle} \def\>{\rangle} \def\GG{\Gamma} \def\gg{\gamma}
  \def\nn{\nabla} \def\pp{\partial} \def\E{\mathbb E}
\def\d{\text{\rm{d}}} \def\bb{\beta} \def\aa{\alpha} \def\D{\scr D}
  \def\si{\sigma} \def\ess{\text{\rm{ess}}}
\def\beg{\begin} \def\beq{\begin{equation}}  \def\F{\scr F}
\def\Ric{\text{\rm{Ric}}} \def\Hess{\text{\rm{Hess}}}
\def\e{\text{\rm{e}}} \def\ua{\underline a} \def\OO{\Omega}  \def\oo{\omega}
 \def\tt{\tilde} \def\Ric{\text{\rm{Ric}}}
\def\cut{\text{\rm{cut}}} \def\P{\mathbb P} \def\ifn{I_n(f^{\bigotimes n})}
\def\C{\scr C}      \def\aaa{\mathbf{r}}     \def\r{r}
\def\gap{\text{\rm{gap}}} \def\prr{\pi_{{\bf m},\varrho}}  \def\r{\mathbf r}
\def\Z{\mathbb Z} \def\vrr{\varrho} \def\ll{\lambda}
\def\L{\scr L}\def\Tt{\tt} \def\TT{\tt}\def\II{\mathbb I}
\def\i{{\rm in}}\def\Sect{{\rm Sect}}  \def\H{\mathbb H}
\def\M{\scr M}\def\Q{\mathbb Q} \def\texto{\text{o}} \def\LL{\Lambda}
\def\Rank{{\rm Rank}} \def\B{\scr B} \def\i{{\rm i}} \def\HR{\hat{\R}^d}
\def\to{\rightarrow}\def\l{\ell}\def\BB{\mathbb B}
\def\8{\infty}\def\I{1}\def\U{\scr U} \def\n{{\mathbf n}}\def\v{V}
\def\V{\mathbb V}
\maketitle

\begin{abstract}
By constructing a suitable  coupling by change of measures, the asymptotic log-Harnack inequality is established for a class of degenerate SPDEs with reflection.
This inequality  implies  the asymptotic heat kernel estimate, the uniqueness of the invariant probability
measure, the asymptotic gradient estimate (hence, asymptotically strong Feller property), and the asymptotic
irreducibility. As application, the main result is illustrated by $d$-dimensional degenerate  stochastic  Navier–Stokes equations with reflection, where the dissipative 
operator is the Dirichlet Laplacian with a power $\theta\ge 1 \lor \ff{d+2}4,$ which includes  the Laplacian when $d\le 2$.
 \end{abstract}


 \vskip 2cm
    \section{Introduction}
   
   Stochastic partial differential equations (SPDEs) with reflection provide a mathematical framework for modeling the evolution of random interfaces in proximity to a hard boundary; see \cite{FO}. The existence and uniqueness of solutions to such reflected stochastic systems were established in \cite{BZ}. For further studies on real-valued SPDEs with reflection, we refer readers to \cite{NP}, \cite{DP}, \cite{XZ}, and the references therein.

On the other hand, to characterize regularity properties of stochastic systems,  the dimension-free Harnack inequality was initiated in \cite{W} for elliptic diffusion semigroup on the Riemannian manifolds, and as a weaker version of this inequality,  the log-Harnack inequality was proposed  in \cite{W2}. Both inequalities have been extensively studied  through the method of coupling by change of measures, see \cite{AT,W3,W4} and the references within for more details. These inequalities lead to important consequences such as gradient estimates, uniqueness of invariant probability measures, heat kernel estimates, and irreducibility of the associated Markov semigroups.
   
 When the noise of a stochastic system is highly degenerate, the above mentioned Harnack type inequalities are not available, so that the asymptotic log-Harnack inequality was introduced in \cite{X} and \cite{BWY} alternatively, for degenerate-noise-driven 2D Navier-Stokes equations and stochastic systems with infinite delay. This inequality also   implies some regularity properties including the asymptotic heat kernel estimate, the uniqueness of the invariant probability
measure, the asymptotic gradient estimate (hence, asymptotically strong Feller property), and the asymptotic
irreducibility, see Theorem \ref{T2.1} for details. 

 In this paper, we establish the asymptotic Harnack inequality for a class of degenerate SPDEs
 with reflection on the unit ball of a Hilbert space,  for which the well-posdeness and  exponential ergodicity have been studied in  \cite{BLZ, BZ}.
   
   Let $(\H, \|\cdot\|_\H, \<\cdot,\cdot\>)$ be a separable Hilbert space, let $(\L_2(\H),\|\cdot\|_{\L_2(\H)})$ denote the spaces of  Hilbert-Schmidt operators on $\H$. 
   Let $(A,\D(A))$ be a positive definite self-adjoint operator on $\H$ with eigenvalues $\{\ll_i>0: i\ge 1\}$ listed in the increasing order counting multiplicities satisfying
   \beq\label{LLM} \lim_{i\to\infty}\ll_i=\infty.\end{equation}
  Let  $\{e_i:i\ge 1\}$ be the corresponding unitary eigenvectors of $\{\ll_i: i\ge 1\}$, which consist of an orthonormal basis of $\H$.
Then   $\V:=D(A^{\frac{1}{2}})$, the domain of $A^{\ff 1 2}$,  with the inner product
$$\<u,v\>_\V:=\<A^{\frac{1}{2}}u,A^{\frac{1}{2}}v\>, \quad u,v\in \V$$
 is  a Hilbert space compactly embedded into $\H$.  
Let $\V^{\ast}$ be  the dual space of $\V$ with respect to $\H$, so that we have a Gelfand triple
$$
\V\hookrightarrow \H\cong \H^{\ast} \hookrightarrow V^{\ast}.$$
Let $ _{\V^*}\<\cdot,\cdot\>_\V$ be the duality between $\V^*$ and $\V$.

    We consider the following SPDE with reflection
   for $X_t^x\in D:= \{x\in \H: \|x\|_\H\le 1\}$:
\begin{equation}\label{SEE}
\beg{split}
&\d X_t^x= \big\{b(X_t^x)+B(X_t^x,X_t^x)-AX_t^x\big\}\d t + \sigma(X_t)\d W_t+\d L_t^x, \\
&t\ge 0,\  X_0^x =x\in {D},
\end{split}
\end{equation}
\beg{enumerate} \item[$\bullet$]  The measurable maps 
$$b: \H\to \V^*,\ \ B: \V\times \V\to \V^*,\ \ \si:\H\to \L_2(\H)$$  are determined by the corresponding physical models in applications. 
\item[$\bullet$]  $W_t$ is the cylindrical Brownian motion on $\H$, i.e.   formally, 
$$W_t=\sum_{i=1}^\infty B_t^i e_i,\ \ t\ge 0$$
for a sequence of independent one-dimensional Brownian motions $\{B^i\}_{i\ge 1}$ on a complete filtered probability space $(\OO, \{\F_t\}_{t\ge 0}, \F, \P)$.
 \item[$\bullet$]    $L_t$ with $L_0=0$ is an adapted  continuous   process on $\H$ with finite variation, i.e. $\P$-a.s.
 $${\rm Var}_H(L)([0,T]):= \sup_{n\ge 1, 0=t_0<t_1\cdots <t_n= T} \sum_{i=1}^n \|L_{t_i}-L_{t_{i-1}}\|_\H<\infty,\  \ T\in (0,\infty).$$
 \end{enumerate}

 Let  $\si_i(x)= \si(x) e_i\in \H$ for $x\in\H$. We have 
 $$\sum_{i=1}^\infty \|\si_i(x)\|_\H^2\le \|\si(x)\|_{\L_2(\H)}^2 \|x\|_\H^2<\infty,\ \ x\in \H.$$

 The rest of the paper is organized as follows. In Section 2, we recall the well-posedness result for \eqref{SEE}.
 In Section 3, we present the main results with complete
proofs. Finally, in Section 4, we apply our main results to $d$-dimensional reflected  Navier-Stokes equations. 

\section{Well-posedness and moment estimates}
 As a preparation,  in this section we recall the definition of solution and the well-posedness result due to \cite{BZ} under the following assumption.
 
\beg{enumerate} \item[{\bf (A)}] \emph{$(A,\D(A))$ is a positive definite self-adjoint operator on $\H$ with eigenvalues $\{\ll_i\}_{i\ge 1}$ satisfying \eqref{LLM},
and $b,B, \si$ satisfy the following conditions.
\item[(1)] There exist constants $K_b,K_\si \in (0,\infty)$ such that for any $x,y\in\H$, 
$$
\|b(x)-b(y)\|_{\V^{\ast}} \le K_b\|x-y\|_\H,\ \  \|\sigma(x)-\sigma(y)\|_{ \L_2(\H)}^2\leq K_\si \|x-y\|_\H^2.$$
\item[(2)] $B: \V\times \V\to \V^*$ is bilinear, and there exists $K_B\in (0,\infty)$ such that   
$$\|B(x,x)\|_{\V^{\ast}}\leq K_B \|x\|_\H\|x\|_\V, \quad\  \ x\in \V.$$
 Moreover,  there exists $K\in (0,\infty)$ such that  $\bar{B} (x,y,z):=\  _{\V^*}\< B(x,y),z\>_{\V}$  satisfies 
\beg{align*}&\bar{B}(x,y,z)= -\bar {B}(x,z,y), \ \ \ \bar{B}(x,y,y)=0,\\
&|\bar{B}(x,y,z)|\le K\|y\|_\V \ss{\|x\|_\H\|x\|_\V\|z\|_\H\|z\|_\V},\ \ \ x,y,z\in \V.\end{align*}}
\end{enumerate} 

 Now we recall the definition of solution introduced  in \cite{BZ}.
\begin{defn}\label{defsol}
A pair $(X_t^x,L_t^x)_{t\ge 0}$ is said to be a solution of the reflected problem $(\ref{SEE})$ iff the following conditions are satisfied
\beg{enumerate} 
\item[(i)] $X_t^x$ is an $\F_t$-adapted    continuous   process  on $\H$ with  $X\in L^2_{loc}([0,\infty);\V)$    $\mathbb{P}$-a.s. 
 \item[(ii)] $L_0=0, L_t$ is a continuous adapted process on $\H$ such that
\begin{equation}
\mathbb{E}[|\text{Var}_H(L)([0,T])|^2]< \infty,\quad\  T\in (0,\infty).\end{equation}  
Moreover, $\P$-a.s. the following Riemann-Stieltjes integral inequality holds:
\begin{equation} \label{def5}
\int_0^T \<\phi(t)-X_t^x,\d L_t\>\geq 0,\ \ \phi \in C([0,T],D).
\end{equation}
\item[(iii)] $\P$-a.s.   the following integral equation in $\V^{\ast}$ holds:
$$X_t^x = x +  \int_0^t \big\{b(X_s^x)+ B(X_s^x,X_s^x)-A X_s^x\big\}\d s +\int_0^t\sigma(X_s^x)\d W_s +L_t^x,\ \ t\ge 0. $$
\end{enumerate}
\end{defn}
The following result is due to  \cite{BZ} and \cite{BLZ}.

\begin{prp}\label{P2.1} Under assumption {\bf (A)}, for any $x\in \H$ the equation $(\ref{SEE}) $ has a unique solution $(X_t^x,L_t^x)$, and  
\begin{equation} \label{esti}
\mathbb{E}\bigg[\sup\limits_{t\in[0,T]}\|X_t^x\|_\H^2 +  \e^{\ll \int_0^T\|X_s^x\|_\V^2\d s}\bigg]<\infty,\ \ \ \ll,T\in (0,\infty).
\end{equation}
\end{prp}

\section{Asymptotic log-Harnack and applications}
 
In this part, we first recall the asymptotic Hanrack inequality and applications, then establish this inequality for the equation \eqref{SEE}.

\subsection{General results}

For a   metric space  $(E, \rr)$, let $P_t$ be a Markov semigroup on $\B_b(E)$, the class of bounded measurable functions on $E$. Let $\B_b^+(E)$ be the set of  strictly positive functions in $\B_b(E)$. 
 
 For any function $f$ on $E$, let $$|\nn f|(x)=\limsup_{y\to x} \ff{|f(x)-f(y)|}{\rr(x,y)},\ \ x\in E.$$   Let  
 $\|\nn f\|_\infty:= \sup_{x\in E} |\nn f|(x),$ and 
 \beq\label{LB} \beg{split}& {\rm Lip}(E):= \{f\in \B_b(E):\ \|\nn f\|_\infty <\infty\},\\
 &\D(E):=\{f\in \B_b^+(E):\ \|\nn\log f\|_\infty<\infty\}.\end{split}\end{equation}  
  
\beg{defn} We call    $P_t$ satisfies the following asymptotic log-Harnack inequality, if there exist symmetric  $\Phi, \Psi_t: E\times E\to (0,\infty)$ with 
$\Psi_t\downarrow 0$ as $t\uparrow \infty$, such that 
\beq\label{ALH}  P_t \log f(x)\le \log P_t f(y) +   \Phi(x,y) + \Psi_t(x,y)\|\nn\log f\|_\infty,\ \ t>0, f\in \D(E).\end{equation} \end{defn}

As shown in \cite{X} that the asymptotic Harnack inequality implies the asymptotically strong Feller property introduced in \cite{HM}. 
A continuous function $\tt\rr : E\times E\to\R_+:=[0,\infty)$ is called a pseudo-metric, provided
$$\tt\rr(x,y)=0\ \text{iff}\ x=y,\ \ \ \tt\rr(x,y)\le
\tt\rr(x,z)+\tt\rr(z,y)\ \text{for}\ x,y,z\in E.$$ For a  pseudo-metric $\tt\rr$ we consider the transportation cost (also called $L^1$-Warsserstein distance when $\tt\rr$ is a distance)
$$\W_1^{\tt\rr} (\mu_1,\mu_2):= \inf_{\pi\in \scr C(\mu_1,\mu_2)} \int_{E\times E} \tt\rr(x,y)\pi(\d x,\d y),\ \ \mu_1,\mu_2\in \scr P(E),$$
where $\scr P(E)$ is the class of probability measures on $E$, and $\scr C(\mu_1,\mu_2)$ is all couplings of $\mu_1$ and $\mu_2$; that is $\pi\in \scr C(\mu_1,\mu_2)$ means that $\pi\in \scr P(E\times E)$ with
$\pi(\cdot\times E)=\mu_1$ and $\pi(E\times\cdot)=\mu_2.$  

An increasing
sequence of pseudo-metrics $\{\rr_n\}_{n=0}^\infty$ (i.e.,
$\rr_i(\cdot,\cdot)\ge \rr_j(\cdot,\cdot), i\ge j$) is called   totally separating   if $\lim_{n\to\infty}\rr_n(x,y)=1$ for all $x\neq
y$.

\begin{defn}[\cite{HM}] The Markov semigroup $P_t$ is called asymptotically strong Feller at a point 
$x\in E,$ if there exist  a totally separating system of
pseudo-metrics $\{\rr_k\}_{k\ge 1}$  and a sequence $t_k\uparrow\infty$
such that
\begin{equation}\label{eq5}
\inf_{U\in\mathcal {U}_x}\limsup_{k\rightarrow\8}\sup_{y\in
U} W_1^{\rr_k}( P_{t_k}(x,\cdot), {P}_{t_k}(y,\cdot))=0,
\end{equation}
where $\mathcal {U}_x$ denotes the collection of all open sets
containing $x$, and $P_t(x, A):= P_t1_A(x)$ for $x\in E$ and measurable $A\subset E$. 
$P_t$  is called asymptotically strong Feller if it is asymptotically strong Feller at any $x\in E$. 
\end{defn}

 The following result is taken from \cite{BWY}. 

\beg{thm}\label{T2.1} Let $P_t$ satisfy $\eqref{ALH}$ for some symmetric $\Phi, \Psi_t: E\times E\to (0,\infty)$  with $\Psi_t\downarrow 0$ as $t\uparrow \infty$.   Then:
\beg{enumerate} \item[$(1)$] If  for any $x\in E$,
\begin{equation}\label{c7}\beg{split}
&\Lambda(x):=\limsup_{y\to
x}\ff{\Phi(x,y)}{\rr(x,y)^2}<\infty,  \\
&\Gamma_t(x)  :=\limsup_{y\to
x}\ff{\Psi_t(x,y)}{\rr(x,y)}<\infty, \end{split}
\end{equation} then  
\beq\label{GES}|\nn  {P}_tf | \le\ss{2\LL}\ss{ 
{P}_tf^2 -(  {P}_tf)^2}+\|\nn
f\|_\8\Gamma_t, \ \ t>0, f\in \Lip(E).\end{equation}
In particular,  when $\Gamma_t\downarrow 0$ as $t\uparrow \infty$, $P_t$ is asymptotic strong Feller. 
 \item[$(2)$]  If $P_t$ has an invariant probability measure $\mu$, then  \begin{equation}\label{03}
\limsup_{t\to\8}P_tf(x)\le
\log\bigg(\ff{\mu(\e^f)}{\int_E\e^{-\Phi(x,y)}\mu(\d
y)}\bigg),~~~~x\in E,\ f\in {\rm Lip}(E).
\end{equation} Consequently, for any closed $A\subset E$ with $\mu(A)=0,$
\begin{equation}\label{05}
\limsup_{t\to\8}P_t1_A(x)=0,~~~~x\in E.
\end{equation}
\item[$(3)$] $P_t$ has at most one   invariant probability measure.
 \end{enumerate}\end{thm} 

\subsection{Main result of the paper}

Let $P_t$ the the Markov semigroup associated with \eqref{SEE}, i.e. 
$$P_t f(x):= \E[f(X_t^x)],\ \ \ t\ge 0,\ x\in\H,\ f\in \B_b(D).$$
To establish \eqref{ALH} for this $P_t$,   we make the following assumption for some $N\in \mathbb N$.
 
\beg{enumerate} \item[${\bf (A_N)}$] \emph{ For any $x\in \H$,
 $$\H_N:= {\rm span}\{e_i: 1\le i\le N\} \subset \si(x) \H:= \{\si(x) z:\ z\in\H\},$$ and there exists a   measurable map
 $$\si^{-1}: \H\times \H_N\to \H$$
 such that  $\si(x)\si^{-1}(x) y=y$ for   any $(x,y)\in\H\times\H_N$, and  
 $$\|\si^{-1}\|_{\H_N} := \sup_{x\in \H, y\in \H_N, \|y\|\le 1} \|\si^{-1}(x)y\| <\infty.  $$ }
 \end{enumerate}
 
 The main result of the paper is the following, which applies to   degenerate noise such that    ${\bf (A_N)}$ holds  for some $N\in \mathbb N$ with  $r(N)>0$. Obviously, we have $r(N)>0$  for large enough $N$   due to  \eqref{LLM}.
 
 \begin{thm}\label{harnack}
Assume {\bf (A)}. If there exists $N\in\mathbb N$ such that  ${\bf (A_N)}$ and $r(N)>0$  hold, where 
$$ r(N):= \lambda_{N+1} -2K_b-3K_\si-2K_B^2(2K_b+\|b(0)\|_{\V^*}^2)- 4K_B^2(4K_B^2+1)(K_\si+\|\si(0)\|_{\L_2(\H)}^2),$$
       then   the asymptotic log-Harnack inequality $\eqref{ALH}$ holds with
\beq\label{PPS}\beg{split}& \Phi(x,y)=  \ff {\e^{K_B^2} \ll_{N+1}^2\|\si^{-1}\|_{\H_N}^2 }{2r(N)} \|x-y\|_\H^2,\\
& \ \ \Psi_t(x,y)=   \e^{\ff 1 2(K_B^2 -r(N)t) }\|x-y\|_\H,\ \ \ t\ge 0,\ x,y\in \H.\end{split}\end{equation}
 Consequently,      $P_t$ has at most one invariant probability measure, and the estimates $\eqref{GES}$, $\eqref{03}$ and $\eqref{05}$ hold for 
 $$\LL(x)= \ff {\e^{K_B^2} \ll_{N+1}^2\|\si^{-1}\|_{\H_N}^2 }{2r(N)} \|x-y\|_\H^2,\ \ \ \GG_t(x)=  \e^{\ff 1 2(K_B^2 -r(N)t) },\ \ \ t\ge 0,\ x\in\H.$$
\end{thm}

\subsection{Proof of Theorem \ref{harnack}}
By Theorem \ref{T2.1} for $E=\H$ and $\rr(x,y)=\|x-y\|_\H$, it suffices to prove \eqref{ALH} for $\Phi$ and $\Psi_t$ given in \eqref{PPS}.
To this end, we will apply the coupling by change of measures.  

\beg{proof}[Proof of Theorem \ref{harnack}] Let
$\pi_N: \H\to \H_N$ be the orthogonal projection, i.e.
$$\pi_N x:= \sum_{i=1}^N \<x,e_i\>e_i,\ \ \ x\in \H.$$
For any $x,y\in {D}$, let $(X^x_t,L^x_t)$ be the solution of equation \eqref{SEE}, and let
  $(Y_t^y,L_t^y)$ be  the solution to the following modified equation:
\begin{equation}\label{4.3}
\beg{split} 
\d Y_t^y =&\,\big\{b(Y_t^y)+B(Y_t^y,Y_t^y) - AY_t^y\big\}\d t +  \sigma(Y_t^x)\d W_t + \d L_t^y\\
& \qquad  +\ff{\lambda_{N+1}}2\pi_N(X_t^x-Y_t^y)\d t,\  \ \ t\geq0,\ Y_0^y=y.
\end{split}\end{equation}

To formulate $P_t f(y)$ using $Y_t^y$, we make use of Girsanov's theorem. Let
\begin{align*}
&\tt W_t:= W (t) + \int_0^t \beta_s \d s,\\
&\beta_s  := \lambda_{N+1}\sigma(Y^y_s)^{-1}\pi_N(X^x_s-Y^y_s).
\end{align*}
By ${\bf (A_N)}$ and $X_t^x,Y_t^y\in D$,  for any $t\geq 0$,
\begin{equation}\label{BT}\beg{split}
&\|\beta_t\|_{\H}\leq \ff 1 2\ll_{N+1}\|\si^{-1}\|_{\H_N} \|\pi_N(X^x_t-Y^y_t)\|\\
&\leq \ff 1 2 \ll_{N+1}\|\si^{-1}\|_{\H_N}\|X^x_t-Y^y_t\|\le  \ll_{N+1}\|\si^{-1}\|_{\H_N}<\infty.\end{split}
\end{equation}
So, by Gisranov's theorem, 
$$R_t:=\e^{-\int_0^t \beta_s \d W_s -\frac{1}{2}\int_0^t\|\bb_s\|_\H^2 \d s},\ \ t\ge 0$$
is a martingale, and for any $t>0$,
$(\tt W_s)_{s\in [0,t]}$ is cylindrical Brownian motion on $\H$ under the weighted probability
$$\d\Q_t:= R_t\d\P.$$ 
So, by the weak uniqueness of \eqref{SEE} for initial value $y$ in place of $x$, and noting that \eqref{4.3} can be reformulated as 
$$
\d Y_t^y =\big\{b(Y_t^y)+B(Y_t^y,Y_t^y) - AY_t^y\big\}\d t +  \sigma(Y_t^x)\d \tt W_t + \d L_t^y,
   \ \ t\geq0,\ Y_0^y=y,$$
  we have  
 $$P_t f(y)= \E_{\Q_t} [f(Y_t^y)],\ \ t>0,\ f\in \B_b(D),$$
 where  $\E_{\Q_t}$ is the expectation with respect to the weighted probability $\Q_t$.
Combining this with Young's inequality \cite[Lemma 2.4]{ATW},  we obtain that for any $f\in \D(D)$, which is defined in \eqref{LB} for $E=D$,
\beg{align*} &\E[\log f(Y_t^y)] = \E_{\Q_t}[R_t^{-1} \log f(Y_t^y)] 
\le \log \E_{\Q_t} [f(Y_t^y)] + \E_{\Q_t}[R_t^{-1}\log R_t^{-1}]\\
&= \log P_t f(y)- \E [\log R_t]=  \log P_t f(y) + \ff 1 2\E \int_0^t \|\bb_s\|_\H^2 \d s.\end{align*}
Therefore,
\beq\label{AF} \beg{split} &P_t \log f(x) = \E[\log f(X_t^x)]=\E[\log f(Y_t^y)]+ \E[\log f(X_t^x)- \log f(Y_t^y)]\\
&\le \log P_tf(y) +  \ff 1 2\E \int_0^t \|\bb_s\|_\H^2 \d s+ \|\nn\log f\|_\infty \E[\|X_t^x-Y_t^y\|],\ \ t>0,\  f\in \D(D).\end{split}\end{equation} 
By Schwarz inequality and Lemma \ref{L2} below,  we obtain
\beq\label{09}\beg{split}&\E\big[\|X_t^x-Y_t^y\|_{\H}^2\big]\\
&\le \Big(\E \left[\e^{-2 K_B^2 \int_0^t\|X^x_s\|^2_\V\d s} \|X^x_t-Y^y_t\|_\H^4\right]\Big)^{\ff 1 2}
 \Big(\E \left[\e^{2K_B^2\int_0^t\|X^x_s\|^2_\V\d s}\right]\Big)^{\ff 1 2}\\
 &\le \e^{K_B^2 -r(N)t} \|x-y\|_\H^2,\ \ \ t\ge 0,\ x,y\in \H.\end{split}\end{equation}
 This together with \eqref{BT} implies 
 $$\ff 1 2\E\int_0^t \|\bb_s\|_\H^2\d s\le \ff {\e^{K_B^2} \ll_{N+1}^2\|\si^{-1}\|_{\H_N}^2 }{2r(N)} \|x-y\|_\H^2.
 $$
 Moreover, by \eqref{09} and  Jensen's inequality,
 $$\E\big[\|X_t^x-Y_t^y\|_{\H}\big]\le \e^{\ff 1 2 (K_B^2-r(N)t)}\|x-y\|_\H.$$  
Therefore,    \eqref{AF}  implies \eqref{ALH} with the  $\Phi$ and $\Psi$ given in \eqref{PPS}.\end{proof}

\begin{lem}\label{L2} Under {\bf (A)} and ${\bf (A_N)}$, for any $t>0$ and $x,y\in \H$, we have 
\begin{equation}\label{T1}
\mathbb{E}\left[ \e^{\ll\int_0^t\|X^x_s\|_\V^2\d s} \right]\leq \e^{\ll +\ll\big[(2K_b+\|b(0)\|_{\V^*}^2)+ (4\ll+2)(K_\si+\|\si(0)\|_{\L_2(\H)}^2)\big]t}, \ \ \ \ll>0,
\end{equation}
 \begin{equation}\label{T2}
\mathbb{E} \left[\e^{-2K_B^2\int_0^t\|X^x_s\|^2_\V\d s} \|X^x_t-Y^y_t\|_\H^4\right]\leq  \e^{(4K_b+6K_\si- 2\lambda_{N+1})t} \|x-y\|_\H^4.
\end{equation}
 \end{lem}
 
\begin{proof}  Below we prove \eqref{T1} and \eqref{T2} respectively. 

(1) By {\bf (A)(2)} we have
$\bar B(x,x,x)=0,$ while \eqref{def5} with $\phi=0$  implies
$$\<X_t^x, \d L_t^x\>\le 0.$$
So, by It\^o's formula, we obtain
\beq\label{IT1} \beg{split} &\d \|X_t\|_\H^2 \le 2\<X_t^x,\si(X_t^x)\d W_t\>\\
&\qquad + \big(\|\si(X_t^x)\|_{\L_2(\H)}^2 +2_{\V^*}\<b(X_t^x), X_t^x\>_\V- 2 \|X_t^x\|_\V^2\big)\d t.\end{split} \end{equation}
Since $X_t^x\in D$ implies $\|X_t^x\|_\H\le 1$, by {\bf (A)}(1)  we obtain
\beq\label{TTB} \beg{split} & {_{\V^*}\<b(X_t^x),X_t^x\>_\V}= {_{\V^*}\<b(X_t^x)-b(0)+b(0), X_t^x\>_\V}\\
&\le  \|b(0)\|_{\V^*} \|X_t^x\|_\V+ K_b\|X_t\|_\H^2\\
&\le 
\ff 1 2 \|b(0)\|_{\V^*}^2 + K_b + \ff 1 2 \|X_t^x\|_{\V}^2,\end{split}\end{equation} and 
\beq\label{TTS} \beg{split} &\|\si(X_t^x)\|_{\L_2(\H)}^2 \le 2 \|\si(0)\|_{\L_2(\H)}^2 + 2 \|\si(X_t^x)-\si(0)\|_{\L_2(\H)}^2 \\
&\le 2 \|\si(0)\|_{\L_2(\H)}^2 + 2K_\si\|X_t^x\|_\H^2\le 2 \|\si(0)\|_{\L_2(\H)}^2 + 2K_\si.\end{split}\end{equation}
Combining \eqref{IT1}-\eqref{TTS} and $\|x\|_\H\le 1 $ for $x\in D$, we obtain 
$$\int_0^t \|X_s^x\|_{\V}^2\d s\le 1+ \big(\|b(0)\|_{\V^*}^2 + 2K_b+2 \|\si(0)\|_{\L_2(\H)}^2 + 2K_\si\big)t+ \int_0^t 2\<X_s^x,\si(X_s^x)\d W_s\>,$$
so  that for any $\ll>0$,
\beq\label{TT2}\E\big[\e^{\ll\int_0^t\|X_s^x\|_\V^2\d s}\big]\le \e^{\ll +\ll \big(\|b(0)\|_{\V^*}^2 + 2K_b+2 \|\si(0)\|_{\L_2(\H)}^2 + 2K_\si\big)t}\ \E\big[\e^{\ll  \int_0^t 2\<X_s^x,\si(X_s^x)\d W_s\>}\big].\end{equation} 
Moreover, \eqref{TTS}  and $\|X_s^x\|_\H\le 1$ imply
\beg{align*}&\E\big[\e^{\ll  \int_0^t 2\<X_s^x,\si(X_s^x)\d W_s\>}\big]\\
&\le \E\big[\e^{\ll  \int_0^t 2\<X_s^x,\si(X_s^x)\d W_s\>-2\ll^2 \int_0^t \|X_s^x\|_\H^2\|\si(X_s^x)\|_{\L_2(\H)}^2\d s} \big] \e^{4\ll^2 \big( \|\si(0)\|_{\L_2(\H)}^2 + K_\si\big)t} \\
&= \e^{4\ll^2 \big( \|\si(0)\|_{\L_2(\H)}^2 + K_\si\big)t}.\end{align*}
Combining this with \eqref{TT2} we derive \eqref{T1}.

(2) Let 
$$g_t=\e^{-2K_B^2\int_0^t\|X^x_s\|_\H^2\d s}.$$ By Ito's formula, we obtain  
$$
\d \big\{g_t\|X^x_t-Y^y_t\|_\H^4\big\} = \d M_t + g_t \big(I_1(t)+I_2(t)+I_3(t)\big)\d t + \d \tt L_t,$$ 
where $M_t$ is a martingale, and   
\beg{align*} &\d\tt L_t:= 4 g_t \|X_t^x-Y_t^y\|_\H^2\<X_t^x-Y_t^y, \d L_t^x-\d L_t^y\>,\\
&I_1(t):=-2K_B^2 \|X_t^x\|_\V^2 \|X_t^x-Y_t^y\|_\H^4 - 4 \|X_t^x-Y_t^y\|_\H^2\|X_t^x-Y_t^y\|_\V^2\\
&\qquad\qquad - 2\ll_{N+1}\|X_t^x-Y_t^y\|_\H^2\|\pi_N(X_t^x-Y_t^y)\|_\H^2,\\
&I_2(t):= 4\|X_t^x-Y_t^y\|_\H^2\,{_{\V^*}\< b(X_t^x)-b(Y_t^y),X_t^x-Y_t^y\>_\V} \\
&\qquad \qquad +4\sum_{i=1}^\infty \<X_t^x-Y_t^y, \si_i(X_t^x)-\si_i(Y_t^y)\>^2 \\
&\qquad \qquad+ 2 \|X_t^x-Y_t^y\|_\H^2\|\si(X_t^x)-\si(Y_t^x)\|_{\L_2(\H)}^2,\\
& I_3(t):= 4\|X_t^x-Y_t^y\|_\H^2 \,  {_{\V^*}\<  B(X_t^x,X_t^x)- B(Y_t^y,Y_t^y), X_t^x-Y_t^y\>_\V}.\end{align*} 
By \eqref{def5}, 
$$ \d \tt L_t\le 0,$$
By {\bf (A)}(1),
$$I_2(t)\le (4K_b+6 K_\si)\|X_t^x-Y_t^y\|_\H^4,$$
while {\bf (A)}(2) implies
$$\<B(y,y),x-y\>=\bar B(y, y, x)= \bar B(y,x,-y)=\bar B(y,x,x-y),\ \ \ x,y,z\in\V,$$ and 
\beg{align*} &|_{\V^*}\<B(x,x)-B(y,y), x-y\>_\V|= |\bar B(x,x,x-y)- \bar B(y,x,x-y)|\\
&\qquad =|_{\V^*}\<B(x-y,x, x-y\>_\V|\le K_B\|x\|_\V\|x-y\|_\V\|x-y\|_\H,\end{align*}
so that 
 \beg{align*}I_3(t)&\le 4K_B\|X_t^x-Y_t^y\|_\H^3 \|X_t^x\|_\V\|X_t^x-Y_t^y\|_\V\\
&\le 2 K_B^2\|X_t^x\|_\V^2 \|X_t^x-Y_t^y\|_\H^4 +  2\|X_t^x-Y_t^y\|_\H^2\|X_t^x-Y_t^y\|_\V^2.\end{align*}
Combining these with $\|X_t^x-Y_t^y\|_\V^2\ge \ll_{N+1} \|(1-\pi_N)(X_t^x-Y_t^y)\|_\H^2,$ we derive 
\beg{align*}&\d \big\{g_t\|X^x_t-Y^y_t\|_\H^4\big\}-\d M_t\\
& \le g_t\Big((4K_b+6K_\si) \|X_t-Y_t\|_\H^4-2  \|X_t^x-Y_t^y\|_\H^2\|X_t^x-Y_t^y\|_\V^2\\
&\qquad\qquad \qquad - 2\ll_{N+1}\|X_t^x-Y_t^y\|_\H^2\|\pi_N(X_t^x-Y_t^y)\|_\H^2\Big)\d t\\
&\le (4K_b+6K_\si- 2\ll_{N+1})  g_t\|X^x_t-Y^y_t\|_\H^4\d t.\end{align*}
Therefore,  by Gronwall's inequality, we prove \eqref{T2}. \end{proof}

\section{Application to reflected stochastic Navier-Stokes equations}
 
In this part, we apply the main result to  the  $d$-dimensional  stochastic Navier-Stokes Equations with reflection, over a  $C^1$ bounded open domain $U\subset \R^d$.

For any $p\in [1,\infty]$, let   $L^p=L^p(U, \mathbb{R}^d)$ be the   space of all $\mathbb{R}^d$-valued measurable functions $v$ 
defined on $U$ with 
$$\|v\|_{L^p}:= \big\||v|\big\|_{L^p(U)}<\infty,$$ where 
$\|\cdot\|_{L^p(U)}$ is the $L^p$-norm with respect to the Lebesgue measure on $U$.

Let $\DD$ be the Dirichlet Laplacian on $U$. For any $m\in\R$ and $q\ge 1$, let $H^{m,q}$ be the closure of $C_0^\infty(U;\R^d),$ 
the space of smooth $\R^d$-valued functions on $U$ with compact support, under the norm
  $$\|u\|_{H^{m,q}}:= \|(-\DD)^{\ff m 2} u\|_{L^q(U)}.$$
 
For $v\in L^1$, we denote ${\rm div} v=0$ if
$$\int_U \<v,\nn f\>(x)\d x=0,\ \ f\in C_0^\infty(U),$$
where $C_0^\infty(U)$ is the set of all smooth functions on $U$ with compact support. 

Let
$$\H:=\big\{u\in L^2:\ {\rm div}u=0\big\},\ \ \<u,v\>:= \<u,v\>_{L^2},\ \ \ u,v\in \H.$$ 
We consider  the following SPDE  on $D:= \{u\in \H:\ \|u\|_\H:=\|u\|_{L^2}\le 1\}$ with refection: 
\beq \label{SEE'}\beg{split}
 &\d X_t^u= \big\{b(X_t^u)+(X_t^u\cdot \nn) X_t^u- \nu (-\DD)^\theta X_t^u\big\}\d t + \sigma(X_t^u)\d W_t+\d L_t^u, \\
 & \ \ \ t\ge 0,\  X_0 =u\in {D},\end{split}
\end{equation} where $\nu,\theta \in (0,\infty)$   are  constants, and $b,\si$ are to be determined latter on.

To apply Theorem \ref{harnack}, we take 
\beq\label{AB}A= \nu (-\DD)^\theta,\ \ \ B(u,v)= (u\cdot\nn)v\ \text{for}\ u,v\in \V,\end{equation}
where $$ \V:=\{u\in H^{\theta,2}: \ {\rm div}u=0\big\},\ \ \|u\|_\V= \|A^{\ff 1 2}u\|_{L^2}=\nu^{\ff 1 2} \|u\|_{H^{\theta,2}}.$$
Let $\V^*$ be the dual space of $\V$ with respect to $\H$, which is the closure of $\H$ with respect to the norm
$$\|u\|_{\V^*}:= \nu^{-\ff 1 2} \|u\|_{H^{-\theta,2}}.$$
It is classical that $(A,\D(A))$   is a positive definite self-adjoint operator on $\H$ with eigenvalues $\{\ll_i\}_{i\ge 1}$ satisfying
$$c_1 i^{\ff {2\theta} d}\le \ll_i\le c_2 i^{\ff {2\theta} d},\ \ i\ge 1$$
 for some constants $c_2>c_1>0,$ so that \eqref{LLM} holds.

Moreover, it is easy to see that 
$$_{\V^*}\<(u\cdot \nn)v, w\>_{\V}= - {_{\V^*}\<(u\cdot \nn)w, v\>_{\V}},\ \ _{\V^*}\<(u\cdot \nn)v, v\>_{\V} =0,\ \ \ \ u,v,w\in \V.$$
So, the following lemma ensures that $B(u,v):= (u\cdot\nn)v$ satisfies {\bf (A)}(2).

\beg{lem}\label{LN1} Let $B(u,v):= (u\cdot\nn)v$ for $u,v\in \V$. If $\theta\ge \ff{d+2}4\lor 1,$ then there exist  constants $C,C_B\in (0,\infty)$ such that
\beq\label{YY}\beg{split}
&|_{H^{-\theta,2}}\<B(u,v), w\>_{H^{\theta,2}}|\le C\|v\|_{H^{\theta,2}}\ss{\|u\|_{L^2}\|u\|_{H^{\theta,2}}\|w\|_{L^2}\|w\|_{H^{\theta,2}}},\\\
&\|B(u,v)\|_{H^{-\theta,2}}\le C_B\|u\|_{L^2}\|v\|_{H^{\theta,2}},\ \ \ u,v,w\in H^{\theta,2}.\end{split}\end{equation}
\end{lem} 

\beg{proof} 
 By the Sobolev inequality,  there exists a constant $c_1\in (0,\infty)$ such that
  \beq\label{SB0} \|u\|_{H^{m,q}}\le c_1\|u\|_{H^{k,p}}\end{equation}
  holds for any $\infty>k\ge m>-\infty$    and $\infty\ge p,q\ge 1$ satisfying
  \beq\label{SB1} \ff 1 q+\ff{k-m}d\ge \ff 1 p.\end{equation}
  
  It is easy to see that \eqref{SB1} holds for
  $$m=\theta,\ \ \ q =2,\ \ \ k= 0,\ \ \ p= \ff{2d}{d+2\theta}\lor 1,$$
  so that \eqref{SB0} implies
  $$\|B(u,v)\|_{H^{-\theta,2}}\le c_1 \|B(u,v)\|_{L^{\ff{2d}{d+2\theta}\lor 1}}.$$
  Combining this with   H\"older's inequality we obtain
$$\|B(u,v)\|_{H^{-\theta,2}}\le c_1 \|u\|_{L^2} \|\nn v\|_{L^{\ff d {\theta}\lor 2}}=c_1 \|u\|_{\H} \|\nn v\|_{L^{\ff d {\theta}\lor 2}}.$$
Noting that 
$$\|\nn v\|_{L^{\ff d {\theta}\lor 2}}\le c_2 \|u\|_{L^2} \|v\|_{H^{1,\ff d {\theta}\lor 2}}$$
  for some constant $c_2\in (0,\infty)$, we derive 
    \beq\label{BV2} \|B(u,v)\|_{H^{-\theta, 2}}\le c_1 \|(u\cdot \nn)v\|_{L^{\ff{2d}{d+2\theta}\lor 1}}\le c_1c_2\|u\|_{\H} \|v\|_{H^{1,\ff d {\theta}\lor 2}}.\end{equation}
  Moreover,   since 
 $\theta\ge  \ff{d+2}4\lor 1$ implies  \eqref{SB1} for
  $$m=1,\ \ \ q=  \ff d {\theta}\lor 2,\ \ \ k= \theta,\ \ \ p=2,$$
  so  by \eqref{SB0} we obtain
  $$ \|v\|_{H^{1,\ff d {\theta}\lor 2}}\le c_1 \|v\|_{H^{\theta,2}}.$$
  Combining this with \eqref{BV2} we derive the second inequality in \eqref{YY}.
  
     Next,   it is clear that 
  $$|_{H^{-\theta,2}}\<B(u,v), w\>_{H^{\theta,2}}|\le   \|B(u,v)\|_{L^2} \|w\|_{L^2} = \|B(u,v)\|_{L^2} \|w\|_\H.$$
  Combining this with the second inequality in \eqref{YY} which is just proved, it remains to find a constant $c\in (0,\infty)$ such that
  \beq\label{**} \|B(u,v)\|_{L^2}\le c \|u\|_{H^{\theta,2}}\|v\|_{H^{\theta,2}}.\end{equation}
  Below we prove this estimate by considering two different situations.
  
  (a) Let $\theta<\ff{d+2}2.$ Then 
 $$q_1:=\ff{2d}{d+2-2\theta}\in [2, \infty),\ \ \ q_2:=\ff{2q_1}{q_1-2}\in (2,\infty]$$
 satisfies $\ff 1 {q_1}+\ff 1 {q_2}=\ff 1 2.$ 
 By the $L^q$-boundedness of Riesz transform $\nn (-\DD)^{-\ff 1 2}$ for $q\in (1,\infty)$, there exists a constant $c_3\in (0,\infty)$ such that  
 $$\|\nn v\|_{L^{q_1}}\le c_3 \|v\|_{H^{1,q_1}}.$$
Combining this with  H\"older's inequality,  we obtain
\beq\label{BW} \|B(u,v)\|_{L^2} \le \|u\|_{L^{q_2}}\|\nn v\|_{L^{q_1}}\le c_3 \|u\|_{H^{0,q_2}}\|v\|_{H^{1,q_1}}.\end{equation}
It is easy to see that
$$\ff 1 {q_1}+\ff{\theta-1}d =\ff 1 2$$
and $\theta\ge \ff{d+2}4$ implies
$$\ff{1}{q_2} +\ff\theta d\ge \ff 1 2.$$
Then \eqref{SB0} holds for $(m,q,k,p)=(0, q_2, \theta, 2)$ or $(1, q_1,\theta, 2)$, so that \eqref{BW} yields
\eqref{**} for some constant $c\in (0,\infty).$

(b) Let $\theta\ge \ff{d+2}2$. Then 
\beq\label{BW'} \|B(u,v)\|_{L^2} \le \|u\|_{L^{4}}\|\nn v\|_{L^{4}}\le c_3 \|u\|_{H^{0,4}}\|v\|_{H^{1,4}}\end{equation}
holds for some constant $c_3\in (0,\infty)$. Noting that 
when $\theta\ge \ff{d+2}2$ the condition $\eqref{SB1}$ holds for
$(m,q,k,p)= (0,4,\theta,2)\ \text{or}\ (1,4,\theta,2)$, we deduce  \eqref{**} from \eqref{BW'} and \eqref{SB0}. 
\end{proof}

Now, let 
$$b: \V\to\V^*\ \ \ \ \si: \H\to \L_2(\H)$$ be measurable satisfying the following assumption.

 \beg{enumerate}\item[{\bf (B)}] \emph{There exists constants $C_b, K_\si\in (0,\infty)$ such that
 \beg{align*}&\|b(x)-b(y)\|_{H^{-\theta,2}}\le C_b \|x-y\|_\H,\\
 &\|\si(x)-\si(y)\|_{\L_2(\H)}^2\le K_\si \|x-y\|_\H^2,\ \ x,y\in \H.\end{align*}}
 \end{enumerate} 
 
Note that
$$\|\cdot\|_\V=\nu^{\ff 1 2}\|\cdot\|_{H^{\theta,2}},\ \ \ \|\cdot\|_{\V^*}=\|\cdot\|_{H^{-\theta,2}}.$$
By Lemma \ref{LN1},  the following result is a direct consequence of Proposition \ref{P2.1} and Theorem \ref{harnack}, which in particular applies to $A=-\DD$ (i.e. $\theta=1$) when $d\le 2$. 
   
 \beg{thm}\label{TB} Assume {\bf (B)} and let $A,B$ be in $\eqref{AB}$ for some constants $\nu\in (0,\infty)$ and $\theta\in [1\lor \ff {d+2}4,\infty)$.
 Then $\eqref{SEE'}$ is well-posed for any initial value $x\in \H$. 
 
 Moreover, let $C_B,C_b$ and $K_\si$ be in Lemma $\ref{LN1} $ {\bf (B)},   let 
 $$K_B= \nu^{-1} C_B,\ \ \ K_b= \nu^{-\ff 1 2} C_b,\ \ \ \|b(0)\|_{\V^*}=\nu^{-\ff 1 2} \|b(0)\|_{H^{-\theta,2}},$$ 
 and let $r(N)$ be defined in  Theorem $\ref{harnack}$.   If there exists $N\in \mathbb N$ such that ${\bf (A_N)}$ and $r(N)>0$ hold, then all assertions 
 in Theorem $\ref{harnack}$ hold for $\eqref{SEE'}$.
  \end{thm}
  
 \begin{rem} By repeating the above  argument for $A:= \nu (1-\DD)^\theta$ on $\mathbb T^d$, where   $\nu\in (0,\infty)$ and $\theta\in[ 1\lor\ff{d+2}4,\infty),$ 
 Theorem $\ref{TB}$ also holds with
 \beg{align*} &\H:= \big\{u\in L^2(\mathbb T^d,\R^d): \ {\rm div} u=0\big\},\\
 &\V:= \big\{u\in H^{\theta,2}(\mathbb T^d,\R^d): \ {\rm div} u=0\big\},\end{align*}
 $\V^*$ being the dual space of $\V$ with respect to $\H$, and $b,\si$ satisfying {\bf (B)}. \end{rem}
   
\noindent{\bf Acknowledgement}. This work is partially supported by National Key R$\&$D program of China (No. 2022 YFA1006001)), 
National Natural Science Foundation of China (Nos. 12531007, 12131019).

\noindent{\bf Data availability}.

No data was used for the search described in the article.

\noindent{\bf Disclosure statement}.

We declare that we have no conflict of interest.

\noindent{\bf declaration of interest}.

The authors do not work for, advise, own shares in, or receive funds from any organization that could benefit from this article, and have declared no affiliations other than their research organizations.

\end{document}